\documentclass[12pt]{iopart}

\usepackage[T1]{fontenc}

\usepackage{amsthm,amssymb}

\usepackage[pdftex]{graphicx}

\newtheorem{proposition}{Proposition}[section]
\newtheorem{theorem}[proposition]{Theorem}
\newtheorem{corollary}[proposition]{Corollary}
\newtheorem{lemma}[proposition]{Lemma}
\newtheorem{observation}[proposition]{Observation}

\theoremstyle{definition}
\newtheorem{definition}[proposition]{Definition}
\newtheorem{example}[proposition]{Example}

\theoremstyle{remark}
\newtheorem{remark}[proposition]{Remark}

\def\BLDA{\mathrm{BLDA}}
\def\BUDA{\mathrm{BUDA}}

\begin{document}

\title[One--sided Diophantine approximations]{One--sided Diophantine approximations}

\author{Jaroslav Han\v{c}l$^{1}$, Ond\v{r}ej Turek$^{1,2,3}$}

\address{$^1$ Department of Mathematics, Faculty of Science, University of Ostrava, 30.\ dubna 22, 701 03 Ostrava, Czech Republic}
\address{$^2$ Department of Theoretical Physics, Nuclear Physics Institute, Czech Academy of Sciences, 250 68 \v{R}e\v{z}, Czech Republic}
\address{$^3$ Laboratory for Unified Quantum Devices, Kochi University of Technology, 782-8502 Kochi, Japan}
\ead{jaroslav.hancl@osu.cz, ondrej.turek@osu.cz}
\vspace{10pt}
\begin{indented}
\item[]September 2018
\end{indented}

\begin{abstract}
The paper deals with best one--sided (lower or upper) Diophantine approximations of the $\ell$-th kind ($\ell\in\mathbb{N}$). We use the ordinary continued fraction expansions to formulate explicit criteria for a fraction $\frac{p}{q}\in\mathbb{Q}$ to be a best lower or upper Diophantine approximation of the $\ell$-th kind to a given $\alpha\in\mathbb{R}$. The sets of best lower and upper approximations are examined in terms of their cardinalities and metric properties. Applying our results in spectral analysis, we obtain an explanation for the rarity of so-called Bethe--Sommerfeld quantum graphs.
\end{abstract}

%
\vspace{2pc}

\section{Introduction}

Diophantine approximations of real numbers is a classical concept in number theory. Its basic idea consists in finding rational numbers with the property of being closer to a given $\alpha\in\mathbb{R}$ than any other rational number with a smaller denominator, in the sense of the following definition.

\begin{definition}\label{Def.DA}
A number $\frac{p}{q}\in\mathbb{Q}$ with $p\in\mathbb{Z}$, $q\in\mathbb{N}$ is called a \emph{best Diophantine approximation of the first kind} to a given $\alpha\in\mathbb{R}$ if
\begin{equation}\label{Diophantine}
\left|\alpha-\frac{p}{q}\right|<\left|\alpha-\frac{p'}{q'}\right|
\end{equation}
holds for all $\frac{p'}{q'}\neq\frac{p}{q}$ such that $p'\in\mathbb{Z}$, $q'\in\mathbb{N}$ and $q'\leq q$. If the inequality (\ref{Diophantine}) is replaced with $|q\alpha-p|<|q'\alpha-p'|$, the corresponding fraction $\frac{p}{q}$ is called a \emph{best Diophantine approximation of the second kind} to the number $\alpha$.
\end{definition}

By their nature, Diophantine approximations are useful as good rational approximations of irrational numbers (recall ancient estimates $22/7$ and $355/113$ for $\pi$). They have also various other remarkable applications, for instance in solving Diophantine equations. Similarly, they are used in the theory of Lagrange numbers and Markoff chains~\cite{Cu89,Ha15,Ha16,PSZ16}, which plays an important role in computer science. 

Recent development in mathematical physics (more specifically, in spectral analysis of periodic quantum graphs \cite{ET17}) led to a need for a mathematical approach that can be referred to as ``best one--sided Diophantine approximations of the $\ell$-th kind'', where $\ell\in\mathbb{N}$. While best Diophantine approximations, introduced in Definition~\ref{Def.DA}, minimize the quantity $\left|\alpha-\frac{p}{q}\right|$ with respect to $q$ within the set of \emph{all} rational numbers $\frac{p}{q}$, best ``one--sided'' Diophantine approximations (of the first kind) aim at minimizing that quantity within the subset of rational numbers with property $\frac{p}{q}\leq\alpha$, or $\frac{p}{q}\geq\alpha$. Let us call such fractions best \emph{lower} Diophantine approximations and best \emph{upper} Diophantine approximations, respectively.

The study of best one--sided Diophantine approximations is related to the theory of asymmetric Diophantine approximations and their precision, which began to develop in the 20th century. Segre~\cite{Se45} demonstrated that each irrational number has infinitely many rational approximations lying within certain asymmetric bounds. Robinson~\cite{Ro47} used continued fractions to provide an alternative proof of Segre's theorem. Another and very short proof was later published by Eggan and Niven~\cite{EN61}. Then Finkelshtein~\cite{Fi93} studied best upper Diophantine approximations of the 2nd kind. He found their characterization in terms of so-called reduced regular continued fractions, the formalism that is described in detail in Perron's book~\cite{Pe1913} and a paper by Zurl~\cite{Zu35}.

It is likely that problems whose solutions rely on the idea of best lower and upper Diophantine approximations of the $\ell$-th kind will re-emerge in physics again in the future, and probably many times. The aim of this paper is thus to establish a relevant theory that could be used in future applications. However, our results are interesting from a purely mathematical point of view as well, as they represent a counterpart to the classical knowledge of standard Diophantine approximations.

Let us emphasize that the sets of best lower and upper Diophantine approximations to a given $\alpha$ cannot be obtained in any simple manner from the set of all best Diophantine approximations given by Definition~\ref{Def.DA}. Indeed, there exist rational numbers that are best lower or best upper Diophantine approximation to $\alpha$, but they do not obey Definition~\ref{Def.DA} (cf.~Example~\ref{Ex.srovnani}). Therefore, the sets of best lower and upper Diophantine approximations need to be constructed anew.

The paper is organized as follows. Sections~\ref{Prelim.onesided} and~\ref{Sect.CF} recall basic facts about Diophantine approximations and continued fractions. In particular, we introduce the notions of best lower and upper Diophantine approximations of the $\ell$-th kind, and derive their elementary properties. Section~\ref{Section:1st,2nd} presents a detailed description of the sets of best lower and upper Diophantine approximations of the first and second kind. In Section~\ref{Section:3rd}, we study best lower and upper aproximations of the third kind. A particular attention is then paid to quadratic irrational numbers (Section~\ref{Sect.quadratic}). Section~\ref{Section:4th} is devoted to best lower and upper approximations of the $\ell$-th kind for $\ell\geq4$. In Section~\ref{Application} we introduce a spectral problem in quantum mechanics that motivates and uses the developed theory. The paper is concluded with a short summary and outlook (Section~\ref{Conclusions}).

Throughout the paper, we use the standard symbols $\mathbb{N}$, $\mathbb{Z}$, $\mathbb{Q}$ and $\mathbb{R}$ for the sets of positive integers, integers, rational numbers and real numbers, respectively. The symbol $\mathbb{N}_0$ denotes the set of nonnegative integers.

\section{Double--sided and one--sided best Diophantine approximations}\label{Prelim.onesided}

Before proceeding to the central notion of this paper (Definition~\ref{Def. rth}), we formulate a natural extension of Definition~\ref{Def.DA}.
\begin{definition}\label{Def. rth abs}
Let $\alpha\in\mathbb{R}$, $\ell\in\mathbb{N}$ and $\frac{p}{q}\in\mathbb{Q}$ for $p\in\mathbb{Z}$, $q\in\mathbb{N}$. We call the number $\frac{p}{q}$ a \emph{best Diophantine approximation of the $\ell$-th kind} to $\alpha$ if
\begin{equation}\label{rth abs}
q^{\ell-1}\left|\alpha-\frac{p}{q}\right|<(q')^{\ell-1}\left|\alpha-\frac{p'}{q'}\right|
\end{equation}
for all $\frac{p'}{q'}\neq\frac{p}{q}$, $p'\in\mathbb{Z}$, $q'\in\mathbb{N}$ and $q'\leq q$.
\end{definition}

Definition~\ref{Def. rth abs} serves as our starting point for introducing one--sided best Diophantine approximations of the $\ell$-th kind. A special case of Definition~\ref{Def. rth} for $\ell=3$ appeared for the first time in~\cite{ET17}; here we consider a general $\ell\in\mathbb{N}$.

\begin{definition}\label{Def. rth}
Let $\alpha\in\mathbb{R}$, $\ell\in\mathbb{N}$ and $\frac{p}{q}\in\mathbb{Q}$ for $p\in\mathbb{Z}$, $q\in\mathbb{N}$. We say that 
\begin{itemize}
\item $\frac{p}{q}$ is a \emph{best lower Diophantine approximation of the $\ell$-th kind} to $\alpha$ if
\begin{equation}\label{rth below}
0\leq q^{\ell-1}\left(\alpha-\frac{p}{q}\right)<(q')^{\ell-1}\left(\alpha-\frac{p'}{q'}\right)
\end{equation}
for all $\frac{p'}{q'}\leq\alpha$ such that $\frac{p'}{q'}\neq\frac{p}{q}$, $p'\in\mathbb{Z}$, $q'\in\mathbb{N}$ and $q'\leq q$.
\item $\frac{p}{q}$ is a \emph{best upper Diophantine approximation of the $\ell$-th kind} to $\alpha$ if
\begin{equation}\label{rth above}
0\leq q^{\ell-1}\left(\frac{p}{q}-\alpha\right)<(q')^{\ell-1}\left(\frac{p'}{q'}-\alpha\right)
\end{equation}
for all $\frac{p'}{q'}\geq\alpha$ such that $\frac{p'}{q'}\neq\frac{p}{q}$, $p'\in\mathbb{Z}$, $q'\in\mathbb{N}$ and $q'\leq q$.
\end{itemize}
\end{definition}

We immediately have the following observation.

\begin{observation}\label{floorceil}
If $\frac{p}{q}$ is a best lower Diophantine approximation of the $\ell$-th kind to $\alpha$, then $p=\lfloor q\alpha\rfloor$. If $\frac{p}{q}$ is a best upper Diophantine approximation of the $\ell$-th kind to $\alpha$, then $p=\lceil q\alpha\rceil$.
\end{observation}

It follows easily from Definition~\ref{Def. rth} that for any $\alpha\in\mathbb{R}$, a fraction $\frac{p}{q}$ is a best lower Diophantine approximation of the $\ell$-th kind to $\alpha$ if and only if $\frac{-p}{q}$ is a best upper approximation of the $\ell$-th kind to $-\alpha$. Therefore, in the rest of the paper we can assume $\alpha\geq0$ without loss of generality.

For the sake of convenience, from now on we will usually drop the adjective ``Diophantine'' in the term ``Diophantine approximation'', and mostly use the following abbreviations:
\begin{itemize}
\item $\BLDA(\ell)$ for ``best lower Diophantine approximation of the $\ell$-th kind'';
\item $\BUDA(\ell)$ for ``best upper Diophantine approximation of the $\ell$-th kind''.
\end{itemize}

Since Definition~\ref{Def. rth} has weaker requirements than Definition~\ref{Def. rth abs}, we obviously have:
\begin{observation}\label{Obs.abs=>one}
If $\frac{p}{q}$ is a best approximation of the $\ell$-th kind to $\alpha$, then $\frac{p}{q}$ is a $\BLDA(\ell)$ or a $\BUDA(\ell)$ to $\alpha$.
\end{observation}
We emphasize, however, that the converse statement is not true. A $\BLDA(\ell)$ or a $\BUDA(\ell)$ to $\alpha$ may not obey Definition~\ref{Def. rth abs}, as we will see in Example~\ref{Ex.srovnani}.

\begin{observation}\label{Obs.l'<l}
If $\frac{p}{q}$ is a best lower (upper) approximation of the $\ell$-th kind to $\alpha$, then $\frac{p}{q}$ is a best lower (respectively, upper) approximation of the $\ell'$-th kind to $\alpha$ for all $\ell'<\ell$.
\end{observation}

\begin{proof}
If $0<q'\leq q$ and
\begin{equation*}
q^{\ell-1}\left|\alpha-\frac{p}{q}\right|<(q')^{\ell-1}\left|\alpha-\frac{p'}{q'}\right|,
\end{equation*}
then obviously
\begin{equation*}
q^{\ell'-1}\left|\alpha-\frac{p}{q}\right|<(q')^{\ell'-1}\left|\alpha-\frac{p'}{q'}\right|
\end{equation*}
for all $\ell'<\ell$. The inequalities above immediately imply that if $p/q$ obeys definition of a best lower (upper) approximation of the $\ell$-th kind to $\alpha$, then it obeys the respective definition for $\ell'$ as well.
\end{proof}

In some situations one can easily specify a certain subset of $\mathbb{Q}$ such that each $\BLDA(\ell)$ (or $\BUDA(\ell)$) to a given $\alpha$ is an element of this subset. We will encounter such situations in subsequent sections. Then the determination of $\BLDA(\ell)$ and $\BUDA(\ell)$ to $\alpha$ can be simplified using Proposition~\ref{Prop.crit} below.

\begin{proposition}\label{Prop.crit}
(i)\; Let $\mathcal{S}_L\subset\mathbb{Q}\cap(-\infty,\alpha]$ contain all $\BLDA(\ell)$ to $\alpha$. Then $\frac{p}{q}\in\mathcal{S}_L$ is a $\BLDA(\ell)$ to $\alpha$ if and only if
\begin{equation}\label{Crit.below}
\forall\,\frac{p'}{q'}\in\mathcal{S}_L:\qquad q'\leq q \quad\Rightarrow\quad q^{\ell-1}\left(\alpha-\frac{p}{q}\right)<(q')^{\ell-1}\left(\alpha-\frac{p'}{q'}\right).
\end{equation}

(ii)\; Let $\mathcal{S}_U\subset\mathbb{Q}\cap[\alpha,\infty)$ contain all $\BUDA(\ell)$ to $\alpha$. Then $\frac{p}{q}\in\mathcal{S}_U$ is a $\BUDA(\ell)$ to $\alpha$ if and only if
\begin{equation*}\label{Crit.above}
\forall\,\frac{p'}{q'}\in\mathcal{S}_U:\qquad q'\leq q \quad\Rightarrow\quad q^{\ell-1}\left(\frac{p}{q}-\alpha\right)<(q')^{\ell-1}\left(\frac{p'}{q'}-\alpha\right)\,.
\end{equation*}
\end{proposition}

\begin{proof}
(i)\; If $\frac{p}{q}$ is a $\BLDA(\ell)$, then (\ref{Crit.below}) is true due to Definition~\ref{Def. rth}.

Conversely, let $\frac{p}{q}\in\mathcal{S}_L$ be \emph{not} a $\BLDA(\ell)$; we will show that $\frac{p}{q}$ violates~(\ref{Crit.below}). Since $\frac{p}{q}$ is not a $\BLDA(\ell)$, there exist $p',q'\in\mathbb{Z}$ such that $\alpha\geq\frac{p'}{q'}\neq\frac{p}{q}$, $0<q'\leq q$ and (\ref{rth below}) is violated, i.e.,
\begin{equation}\label{smaller}
q^{\ell-1}\left(\alpha-\frac{p}{q}\right)\geq(q')^{\ell-1}\left(\alpha-\frac{p'}{q'}\right).
\end{equation}
Among the pairs $(p',q')$ with this property, choose the pair for which the quantity $(q')^{\ell-1}\left(\alpha-\frac{p'}{q'}\right)$ is minimal. In case that several such pairs exist, let us consider the one with minimal $q'$. This construction guarantees that the fraction $\frac{p'}{q'}$ is a $\BLDA(\ell)$ to $\alpha$. Hence $\frac{p'}{q'}\in\mathcal{S}_L$, and (\ref{Crit.below}) is violated due to~(\ref{smaller}).

(ii)\; The proof is similar to (i).
\end{proof}

\section{Continued fractions}\label{Sect.CF}

Any $\alpha\in\mathbb{R}$ can be expressed in terms of a \emph{continued fraction}, that is, in the form
\begin{equation}\label{contfrac}
\alpha=a_0+\frac{1}{a_1+\frac{1}{a_2+\frac{1}{a_3+\frac{1}{\cdots}}}}\;,
\end{equation}
where $a_0\in\mathbb{Z}$ and $a_j\in\mathbb{N}$ for all $j\in\mathbb{N}$. The fraction on the right hand side of~(\ref{contfrac}) is commonly represented using the notation $[a_0;a_1,a_2,a_3,\ldots]$.

It is easy to see that the sequence $a_0,a_1,a_2,\ldots$ is finite if and only if $\alpha\in\mathbb{Q}$. For finite continued fractions $\alpha=[a_0;a_1,a_2,\ldots,a_n]$ ($n\in\mathbb{N}$), we usually assume that the last term $a_n$ is different from $1$ for the sake of uniqueness of the representation~(\ref{contfrac}) (notice that $[a_0;a_1,a_2,\ldots,a_{n-1},1]=[a_0;a_1,a_2,\ldots,a_{n-1}+1]$).

For a given continued fraction $\alpha=[a_0;a_1,a_2,a_3,\ldots]$ and $n\in\mathbb{N}_0$, the fraction $\frac{p_n}{q_n}=[a_0;a_1,a_2,\ldots,a_n]$ is called the \emph{$n$-th convergent} of $\alpha$. If we set
\begin{equation*}
p_{-1}=1,\quad p_0=a_0 \qquad\mbox{and}\qquad q_{-1}=0,\quad q_0=1,
\end{equation*}
the values of $p_n$ and $q_n$ ($n\in\mathbb{N}$) are given by recurrent formulas
\begin{equation}\label{rekurence}
p_{n}=a_{n}p_{n-1}+p_{n-2} \qquad\mbox{and}\qquad q_{n}=a_{n}q_{n-1}+q_{n-2}\,.
\end{equation}
Numbers $p_n$ and $q_n$ obey the following well-known rules~\cite[eq.~(8) and Thm.~6]{Kh64}:
\begin{equation}\label{Kh(8)}
q_n p_{n-1}-p_n q_{n-1}=(-1)^n \qquad\mbox{for all $n\geq0$},
\end{equation}
\begin{equation}\label{Kh Thm6}
\frac{q_{n}}{q_{n-1}}=[a_n;a_{n-1},\ldots,a_1] \qquad\mbox{for all $n\geq1$}.
\end{equation}

The recurrent formulas~(\ref{rekurence}) remain valid also if the terms $a_n>0$ in~(\ref{contfrac}) are not integers~\cite[\S3]{Sch80}. This will help us to derive an important identity in Proposition~\ref{Prop.Ha} below.
\begin{proposition}\label{Prop.Ha}
For every $n\geq1$, we have
\begin{equation}\label{Ha.(2.2)}
\alpha-\frac{p_n}{q_n}=\frac{(-1)^n}{q_n^2\left([a_{n+1};a_{n+2},\ldots]+[0;a_n,a_{n-1},\ldots,a_1]\right)}\;.
\end{equation}
\end{proposition}

\begin{proof}
Since
$$
\alpha=[a_0;a_1,a_2,\ldots,a_n,a_{n+1},a_{n+2},\ldots]=[a_0;a_1,a_2,\ldots,a_n,[a_{n+1};a_{n+2},\ldots]]\,,
$$
formula~(\ref{rekurence}) gives
\begin{equation}\label{krok1}
\alpha-\frac{p_n}{q_n}=\frac{p_n[a_{n+1};a_{n+2},\ldots]+p_{n-1}}{q_n[a_{n+1};a_{n+2},\ldots]+q_{n-1}}-\frac{p_n}{q_n}\,.
\end{equation}
Applying~(\ref{Kh(8)}), we transform~(\ref{krok1}) into
\begin{equation}\label{krok2}
\fl \alpha-\frac{p_n}{q_n}=\frac{(-1)^n}{q_n\left(q_n[a_{n+1};a_{n+2},\ldots]+q_{n-1}\right)}=\frac{(-1)^n}{q_n^2\left([a_{n+1};a_{n+2},\ldots]+\frac{q_{n-1}}{q_n}\right)}\;.
\end{equation}
Finally, (\ref{Kh Thm6}) and trivial identity $[a_n;a_{n-1},\ldots,a_1]^{-1}=[0;a_n,a_{n-1},\ldots,a_1]$ allows us to rearrange the denominator on the right hand side of~(\ref{krok2}) into the required form
$
q_n^2\left([a_{n+1};a_{n+2},\ldots]+[0;a_n,a_{n-1},\ldots,a_1]\right).
$
\end{proof}

A \emph{semiconvergent} (or \emph{intermediate fraction}) of $\alpha$ is a fraction of the form
\begin{equation}\label{semiconvergents}
\frac{p_{n}r+p_{n-1}}{q_{n} r+q_{n-1}}\;, \qquad \mbox{where $0<r<a_{n+1}$}.
\end{equation}
Note that if we set $r=0$ (except for $n=0$) or $r=a_{n+1}$ in~(\ref{semiconvergents}), we get the convergents $\frac{p_{n-1}}{q_{n-1}}$ and $\frac{p_{n+1}}{q_{n+1}}$, respectively.

Let us resume well--known facts about values of convergents and semiconvergents:
\begin{proposition}\label{Prop.monotonie}
\begin{itemize}
\item \textup{\cite[Thm.~4 and Thm.~8]{Kh64}} The even-order convergents are smaller or equal to $\alpha$ and form an increasing sequence. The odd-order convergents are greater or equal to $\alpha$ and form a decreasing sequence. That is,
\begin{equation*}
\frac{p_0}{q_0}<\frac{p_2}{q_2}<\frac{p_4}{q_4}<\cdots\leq\alpha\leq\cdots<\frac{p_5}{q_5}<\frac{p_3}{q_3}<\frac{p_1}{q_1}\;.
\end{equation*}
\item \textup{\cite[p.~13--14]{Kh64}} The fractions
$$
\frac{p_{n-2}}{q_{n-2}}=\frac{p_{n-1}\cdot0+p_{n-2}}{q_{n-1}\cdot0+q_{n-2}}\,,\;\frac{p_{n-1}\cdot1+p_{n-2}}{q_{n-1}\cdot1+q_{n-2}}\,,\;\cdots\,,\;\frac{p_{n-1}(a_{n}-1)+p_{n-2}}{q_{n-1}(a_{n}-1)+q_{n-2}}\,,\;\frac{p_{n-1} a_{n}+p_{n-2}}{q_{n-1} a_{n}+q_{n-2}}=\frac{p_{n}}{q_{n}}
$$
form a monotonous sequence that is increasing for even $n$ and decreasing for odd $n$.
\end{itemize}
\end{proposition}

\noindent Continued fractions are compared using the following criterion:
\begin{proposition}\label{Prop.comparison}
(i)\; Let $\alpha=[a_0;a_1,a_2,a_3,\ldots]$, $\beta=[b_0;b_1,b_2,b_3,\ldots]$ and $n$ be the minimal index such that $a_n\neq b_n$. Then
$$
\alpha<\beta \quad\Leftrightarrow\quad (\mbox{$n$ is even} \;\mbox{and}\; a_n<b_n) \;\;\mbox{or}\;\; (\mbox{$n$ is odd} \;\mbox{and}\; a_n>b_n).
$$
(ii)\; If $\alpha=[a_0;a_1,a_2,\ldots,a_n]$ and $\beta=[a_0;a_1,a_2,a_3,\ldots]$, then $\alpha<\beta$ if and only if $n$ is even.
\end{proposition}
\begin{proof}
Proposition \ref{Prop.comparison} is an immediate consequence of Proposition \ref{Prop.Ha}: we write $\alpha -\beta =\alpha -\frac{p_{n-1}}{q_{n-1}}-(\beta -\frac{p_{n-1}}{q_{n-1}})$ and apply (\ref{Ha.(2.2)}) on both expressions $\alpha -\frac{p_{n-1}}{q_{n-1}}$ and $\beta -\frac{p_{n-1}}{q_{n-1}}$.
\end{proof}

\section{Approximations of the first and second kind}\label{Section:1st,2nd}

We provide a complete characterization of best lower Diophantine approximations and best upper Dipohantine approximations of the first and second kind in this section.

We start from a necessary condition for $\frac{p}{q}$ to be a best one--sided approximation of the \emph{first kind} to a given $\alpha$.\footnote{A statement equivalent to Theorem~\ref{Thm.1} was recently published independently by S.~Bettin in~\cite{Be17}.}

\begin{theorem}\label{Thm.1}
Every best lower or upper approximation of the 1st kind to $\alpha$ is either a convergent or a semiconvergent of $\alpha$.
\end{theorem}

\begin{proof}
We will prove that every best lower approximation of the $1$-st kind to $\alpha$ is a convergent or a semiconvergent of $\alpha$. The case of best upper approximations would be treated similarly, so we omit it for the sake of brevity.

To prove this, we assume that $\frac{p}{q}<\alpha$ ($p\in\mathbb{Z}$, $q\in\mathbb{N}$) is neither a convergent nor a semiconvergent of $\alpha$, and show that $\frac{p}{q}$ is not a $\BLDA(1)$ to $\alpha$.
Proposition~\ref{Prop.monotonie} implies that the smallest convergent or semiconvergent of $\alpha$ is $\frac{p_0}{q_0}=\frac{a_0}{1}$. The proof thus falls into $2$ cases:

$\bullet$\; If $\frac{p}{q}<\frac{a_0}{1}$, we have
\begin{equation*}
\alpha-\frac{p}{q}>\alpha-\frac{a_0}{1} \qquad \mbox{and} \qquad 0<1\leq q\,;
\end{equation*}
i.e., $\frac{p}{q}$ contradicts (\ref{rth below}) (consider $p'=a_0,q'=1$). So $\frac{p}{q}$ is not a $\BLDA(1)$.

$\bullet$\; Let $\frac{p}{q}<\alpha$ lie between two adjacent fractions from the set of convergents and semiconvergents. That is, due to Proposition~\ref{Prop.monotonie}, $\frac{p}{q}$ satisfies
\begin{equation*}
\frac{p_{n}r+p_{n-1}}{q_{n}r+q_{n-1}}<\frac{p}{q}<\frac{p_{n}(r+1)+p_{n-1}}{q_{n}(r+1)+q_{n-1}}
\end{equation*}
for some odd $n$ and $r\in\{0,1,\ldots,a_{n+1}-1\}$. Furthermore,
\begin{equation}\label{odhad shora}
\eqalign{
\fl \frac{p}{q}-\frac{p_{n}r+p_{n-1}}{q_{n}r+q_{n-1}}
&<\frac{p_{n}(r+1)+p_{n-1}}{q_{n}(r+1)+q_{n-1}}-\frac{p_{n}r+p_{n-1}}{q_{n} r+q_{n-1}}
=\frac{p_{n}q_{n-1}-q_{n}p_{n-1}}{\left(q_{n}r+q_{n-1}\right)\left(q_{n}(r+1)+q_{n-1}\right)} \\
&=\frac{1}{\left(q_{n}r+q_{n-1}\right)\left(q_{n}(r+1)+q_{n-1}\right)}
}
\end{equation}
(in the last step, we used~(\ref{Kh(8)}) together with the odd parity of $n$).
At the same time, we have
\begin{equation}\label{odhad zdola}
\frac{p}{q}-\frac{p_{n}r+p_{n-1}}{q_{n}r+q_{n-1}}
=\frac{p(q_{n}r+q_{n-1})-q(p_{n} r+p_{n-1})}{q\left(q_{n}r+q_{n-1}\right)}
\geq\frac{1}{q\left(q_{n}r+q_{n-1}\right)}\,.
\end{equation}
Combining estimates (\ref{odhad shora}) and (\ref{odhad zdola}), we get
\begin{equation*}\label{odhad q}
q_{n} (r+1)+q_{n-1}<q\,.
\end{equation*}
Therefore, considering $p'=p_{n} (r+1)+p_{n-1}$ and $q'=q_{n} (r+1)+q_{n-1}$, we conclude that $\frac{p}{q}<\alpha$ contradicts (\ref{rth below}). Hence $\frac{p}{q}$ is not a $\BLDA(1)$ to $\alpha$.
\end{proof}

Theorem~\ref{Thm.1} with Observation~\ref{Obs.l'<l} has the following corollary.

\begin{corollary}\label{Coro.1}
For all $\ell\geq1$, every $\BLDA(\ell)$ and $\BUDA(\ell)$ to $\alpha$ is a convergent or a semiconvergent of $\alpha$.
\end{corollary}

In the next step, we find a sufficient condition for best one--sided approximations of the second kind.

\begin{theorem}\label{Thm.2}
Every convergent and semiconvergent of $\alpha$ is a best lower or upper approximation of the 2nd kind to $\alpha$.
\end{theorem}

\begin{proof}
From Corollary~\ref{Coro.1} we obtain that the only possible candidates for best one--sided approximations of the second kind to $\alpha$ are the fractions
\begin{equation}\label{candidate}
\frac{p_{n}r+p_{n-1}}{q_{n}r+q_{n-1}}
\end{equation}
where $n\in\mathbb{N}_0$ and $r\in\{0,1,\ldots,a_{n+1}-1\}$.
Furthermore, with regard to Proposition~\ref{Prop.monotonie}, number $n$ takes odd values for $\BLDA(2)$ and even values for $\BUDA(2)$. Let us focus on odd $n$; the case of even $n$ is similar.

We will use Proposition~\ref{Prop.crit}(i) where $\mathcal{S}_L=\left\{\frac{p_{n}r+p_{n-1}}{q_{n}r+q_{n-1}}\;:\;\mbox{$n$ is odd}\right\}$. Our goal is to show that all elements of $\mathcal{S}_L$ are $\BLDA(2)$ to $\alpha$. With regard to condition~(\ref{Crit.below}), we will prove that if we arrange the elements of $\mathcal{S}_L$ in a sequence with growing denominators, then the quantities
\begin{equation}\label{quant}
\left(q_{n}r+q_{n-1}\right)\left(\alpha-\frac{p_{n}r+p_{n-1}}{q_{n}r+q_{n-1}}\right)
\end{equation}
strictly decrease.

For a given $n$, the denominators $q_{n}r+q_{n-1}$ obviously grow as $r$ grows from $0$ to $a_{n+1}-1$. Furthermore, for the choice $r=a_{n+1}$ we have $q_{n}a_{n+1}+q_{n-1}=q_{n+1}=q_{n+2}\cdot0+q_{n+1}$. In other words, taking $r=a_{n+1}$ for a given $n$ is equivalent to increasing $n$ by $2$ (i.e., to the next odd value) and taking $r=0$. Consequently, if we arrange the elements of $\mathcal{S}_L$ according to their denominators, then any two consecutive elements can be written as
\begin{equation*}
\frac{p_{n}r+p_{n-1}}{q_{n}r+q_{n-1}} \quad\mbox{and}\quad \frac{p_{n}(r+1)+p_{n-1}}{q_{n}(r+1)+q_{n-1}}
\end{equation*}
for some odd $n$ and $r\in\{0,1,\ldots,a_{n+1}-1\}$.
The monotony of the quantities~(\ref{quant}) in terms of the denominators $q_{n}r+q_{n-1}$ can be thus verified by proving the inequality
\begin{equation}\label{podminka.2-}
\fl \left(q_{n}(r+1)+q_{n-1}\right)\left(\alpha-\frac{p_{n}(r+1)+p_{n-1}}{q_{n}(r+1)+q_{n-1}}\right)<\left(q_{n}r+q_{n-1}\right)\left(\alpha-\frac{p_{n}r+p_{n-1}}{q_{n}r+q_{n-1}}\right)
\end{equation}
for every odd $n$ and $r\in\{0,1,\ldots,a_{n+1}-1\}$.
A straightforward manipulation leads to a simplification of~(\ref{podminka.2-}) to
\begin{equation}\label{podminka.2-.}
\alpha q_{n}-p_{n}<0\,.
\end{equation}
Since $n$ is odd, we have $\frac{p_{n}}{q_{n}}>\alpha$ (see Prop.~\ref{Prop.monotonie}); thus inequality~(\ref{podminka.2-.}) holds true.
\end{proof}

Theorem~\ref{Thm.2} together with Observation~\ref{Obs.l'<l} for $\ell=2$ and $\ell'=1$ imply:
\begin{corollary}\label{Coro.2}
Every convergent or a semiconvergent of $\alpha$ is either a $\BLDA(1)$ or a $\BUDA(1)$ to $\alpha$.
\end{corollary}

Now we are ready to give a complete description of the set of best lower and upper approximations, both of the first and the second kind:

\begin{theorem}\label{Thm.12}
Let $\alpha=[a_0;a_1a_2,a_3,\ldots]\in\mathbb{R}$. For every $n\in\mathbb{N}_0$, let $\frac{p_n}{q_n}$ be the $n$-th convergent of $\alpha$.
\begin{itemize}
\item[(i)] The set of best lower approximations of the 1st kind to $\alpha$ is equal to the set of best lower approximations of the 2nd kind to $\alpha$. Both the sets consist of fractions
\begin{equation}\label{pq12}
\frac{p_{n}r+p_{n-1}}{q_{n} r+q_{n-1}} \qquad (0\leq r< a_{n+1})
\end{equation}
where $n$ is odd.
\item[(ii)] The set of best upper approximations of the 1st kind to $\alpha$ is equal to the set of best upper approximations of the 2nd kind to $\alpha$. Both the sets consist of fractions (\ref{pq12}) for an even $n$, except for the pair $(n,r)=(0,0)$.
\end{itemize}
\end{theorem}

\begin{proof}
Let $\ell$ be $1$ or $2$. Corollary~\ref{Coro.2} implies that every best lower or upper approximations of the $\ell$-th kind to $\alpha$ has form~(\ref{pq12}). Conversely, each fraction~(\ref{pq12}) is a $\BLDA(\ell)$ or a $\BUDA(\ell)$ to $\alpha$ due to Corollary~\ref{Coro.1}. Finally, from Proposition~\ref{Prop.monotonie} we obtain that odd numbers $n$ in~(\ref{pq12}) correspond to $\BLDA(\ell)$, and even numbers $n$ correspond to $\BUDA(\ell)$.
\end{proof}

Let us compare our results on best one--sided approximations to classical results on ``double--sided'' best approximations. It is well known that:
\begin{itemize}
\item The set of best approximations of the first kind to an $\alpha$ consists of all convergents of $\alpha$ (except for $\frac{p_0}{q_0}$ when $\alpha=a_0+\frac{1}{2}$) and \emph{some} semiconvergents. \cite[Thm.~15]{Kh64}
\item Fraction $\frac{p}{q}$ is a best approximation of the second kind to the number $\alpha$ if and only if $\frac{p}{q}$ is a convergent of $\alpha$, except for $\frac{p_0}{q_0}$ when $\alpha=a_0+\frac{1}{2}$. \cite[Thm.~16 and 17]{Kh64}
\end{itemize}
By contrast, as we found in Theorem~\ref{Thm.12}, the set of all one--sided best approximations of the first kind and the set of all one--sided best approximations of the second kind both coincide with the set of all convergents and semiconvergents of $\alpha$.
We illustrate the result with an example.

\begin{example}\label{Ex.srovnani}
If $\alpha=\pi=[3;7,15,1,292,1,\ldots]$, fractions~(\ref{pq12}) for $n=0$ and $r=1,\ldots,7$ are
\begin{equation}\label{ex.pi}
\frac{4}{1}\,,\quad\frac{7}{2}\,,\quad\frac{10}{3}\,,\quad\frac{13}{4}\,,\quad\frac{16}{5}\,,\quad\frac{19}{6}\,,\quad\frac{22}{7}\,.
\end{equation}
Using Definition~\ref{Def.DA}, it is easy to check that among the fractions listed in~(\ref{ex.pi}), only $\frac{13}{4},\,\frac{16}{5},\,\frac{19}{6},\,\frac{22}{7}$ are best approximations to $\pi$ of the 1st kind, and only the fraction $\frac{22}{7}$ is a best approximation to $\pi$ of the second kind. But all the fractions~(\ref{ex.pi})---and no other with denominator $q\leq7$---are $\BUDA(1)$. The same is true for $\BUDA(2)$.
\end{example}

\begin{remark}\label{Geometric}
Best lower and upper Diophantine approximations of the 2nd kind (which coincide with one--sided approximations of the 1st kind due to Theorem~\ref{Thm.12}) have a nice geometric interpretation, see Figure~\ref{Fig.geometry}.
Consider the graph of linear function $f(x)=\alpha x$ and a grid of points $[x,y]$ with integer coordinates.
For each point $[x,y]$ of the grid, one can measure its vertical distance to the graph of $f(x)$.
Then a fraction $\frac{p}{q}$ for $p\in\mathbb{Z}$, $q\in\mathbb{N}$ is a $\BLDA(2)$ to $\alpha$ if and only if $[q,p]$ lies on or below the graph of $f(x)$ and its vertical distance to the graph of $f(x)$ is smaller than the vertical distance between the graph and any other point $[q',p']$ of the grid lying on or below the graph and having coordinate $0<q'\leq q$.
In other words, the point $[q,p]$ has smaller vertical distance from the graph of $f(x)$ than any other point $[q',p']\neq[0,0]$ of the grid lying in the triangle with vertices $[0,0]$, $[q,0]$ and $[q,f(q)]$.
\begin{figure}[h]
\begin{center}
\begin{tabular}{ccc}
\includegraphics[width=0.3\columnwidth]{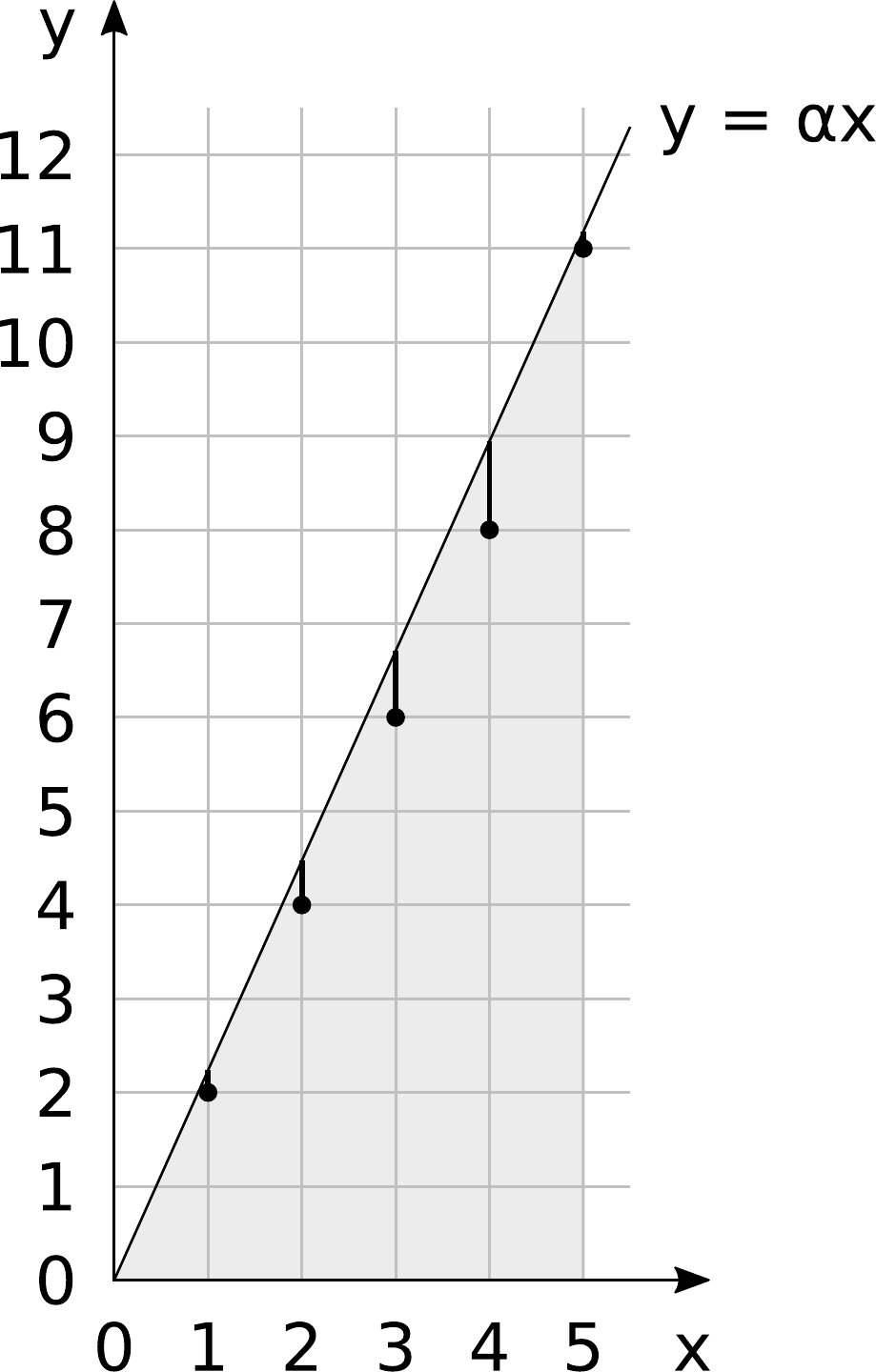}
& \qquad\qquad\qquad &
\includegraphics[width=0.3\columnwidth]{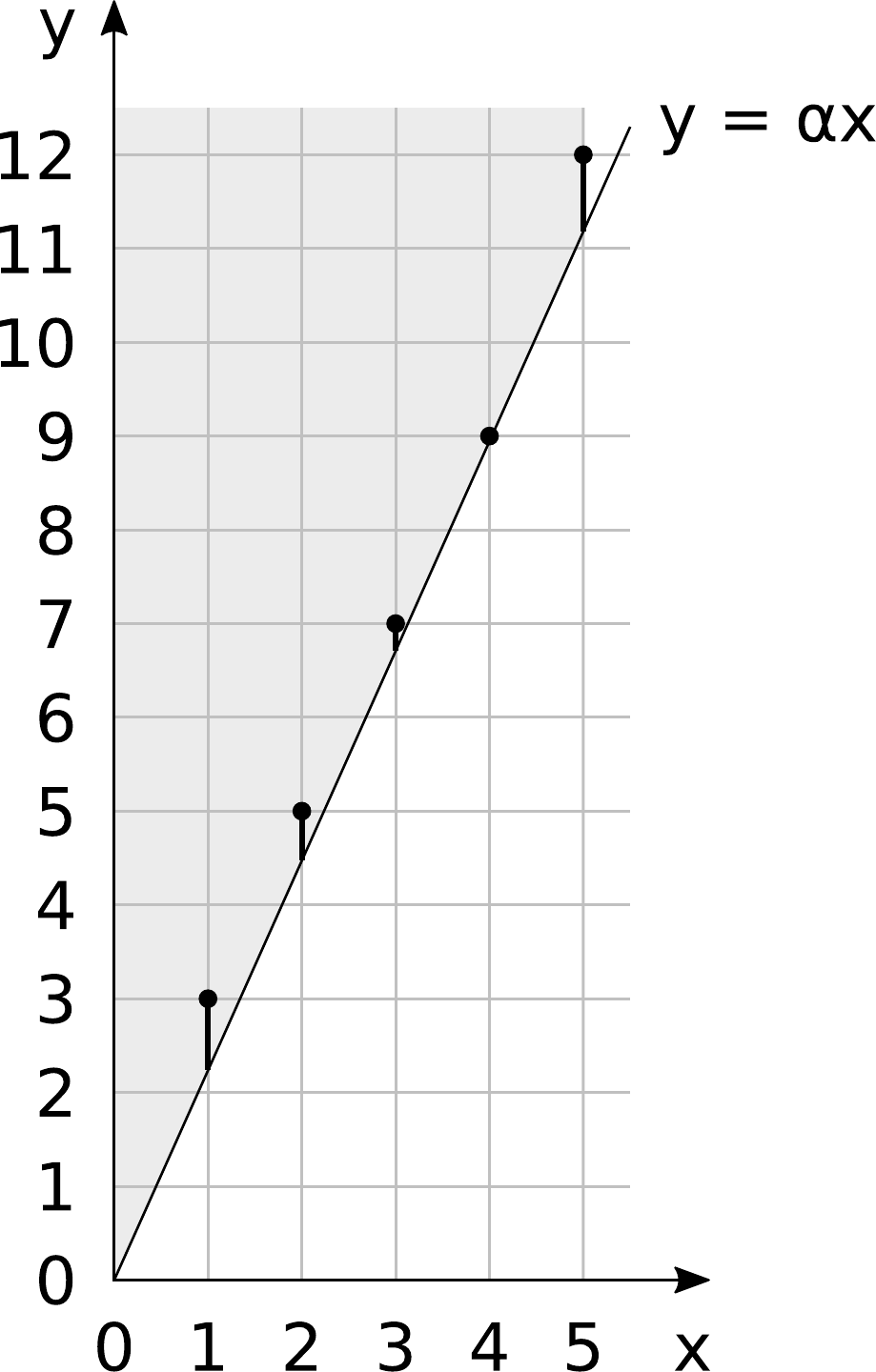}
\smallskip
\\
(a)\; & & (b)\;
\end{tabular}
\end{center}
\caption{Geometric meaning of best lower \emph{(a)} and upper \emph{(b)} approximations of the 2nd kind (plotted for $\alpha=\sqrt{5}$). Regarding $\BLDA(2)$ (Figure (a)), take grid points $[q,p]\in\mathbb{N}^2$ that lie immediately below the graph of $f(x)=\alpha x$, i.e., $[1,2]$, $[2,4]$, $[3,6]$, $[4,8]$ and $[5,11]$. Their vertical distances to the graph are approximately $0.24$, $0.47$, $0.71$, $0.94$ and $0.18$, respectively. The minimal distance with respect to $0<q'\leq q$ is thus attained for $q=1$ and $q=5$. Hence $\frac{2}{1}$ and $\frac{11}{5}$ are the only $\BLDA(2)$ to $\alpha=\sqrt{5}$ among all fractions having denominators $q\leq5$. For $\BUDA(2)$ (Figure (b)), consider grid points that lie immediately above the graph of $f(x)=\alpha x$. Their vertical distances to the graph are approximately $0.76$, $0.53$, $0.29$, $0.06$ and $0.82$, respectively. The minimality with respect to $0<q'\leq q$ is attained for values $q=1,2,3,4$ and \emph{not} for $q=5$; hence $\frac{3}{1}$, $\frac{5}{2}$, $\frac{7}{3}$ and $\frac{9}{4}$ are $\BUDA(2)$ to $\sqrt{5}$, while $\frac{12}{5}$ is not.}
\label{Fig.geometry}
\end{figure}

Similarly, $\frac{p}{q}$ is a $\BUDA(2)$ to $\alpha$ if and only if $[q,p]$ lies on or above the graph of $f(x)$ and its vertical distance to the graph of $f(x)$ is than the vertical distance between the graph and any other point $[q',p']$ of the grid lying on or above the graph and having coordinate $0<q'\leq q$.
\end{remark}

\begin{remark}
We were notified by a referee that the results presented in this section are to some extent known among number theorists in connection with other problems. This concerns in particular the structure of best one--sided approximations of the second kind. But it is not simple to find them with proofs in the literature.

Furthermore, there exists an alternative characterization of the set of best upper approximations of the second kind, which was obtained by Y.~Y.~Finkelshtein within the context of so-called Klein polygons\footnote{We thank the referee for pointing our attention to that result.}. The approximations are expressed in terms of \emph{reduced regular continued fractions}, instead of ordinary regular continued fractions that are used in the present paper. However, the only accessible material on Finkelshtein's result regarding $\BUDA(2)$ seems to be a short note~\cite{Fi93} where no proofs are provided.
\end{remark}

\section{Approximations of the third kind}\label{Section:3rd}

\begin{theorem}\label{Thm.3} We have:
\begin{itemize}
\item[(i)] Every best lower approximation of the 3rd kind to $\alpha$ is an even--order convergent of $\alpha$.
\item[(ii)] Every best upper approximation of the 3rd kind to $\alpha$ is either $\frac{\lceil\alpha\rceil}{1}$ or an odd--order convergent of $\alpha$.
\end{itemize}
\end{theorem}

\begin{remark}
The first version of Theorem~\ref{Thm.3} appeared in~\cite[Prop.~3.5 and 3.6]{ET17}, but the proof there turns out to be mistaken\footnote{The argument given in~\cite[Prop.~3.5]{ET17} relies on Lemma~3.4 ibidem. However, there is a misprint in \cite[Lemma~3.4]{ET17}, namely, the term $a_n$ should read as $a_{n+1}$ everywhere in its formulation and proof ($4$ occurrences). The dependence on a mistaken lemma makes the proof of~\cite[Prop.~3.5]{ET17} invalid.}.
\end{remark}

\begin{proof}[Proof of Theorem~\ref{Thm.3}]
(i)\; Let $\alpha=[a_0;a_1,a_2,\ldots]\in\mathbb{R}$. Due to Theorem~\ref{Thm.12} and Observation~\ref{Obs.l'<l}, each $\BLDA(3)$ to $\alpha$ is given as
\begin{equation}\label{semic}
\frac{p_{n}r+p_{n-1}}{q_{n}r+q_{n-1}}
\end{equation}
for some odd $n\in\mathbb{N}$ and $r$ satisfying $0\leq r< a_{n+1}$. We shall show that if fraction~(\ref{semic}) is a semiconvergent, i.e., if $r$ satisfies $0<r<a_{n+1}$, then (\ref{semic}) is not a $\BLDA(3)$ to $\alpha$.
To prove this, we will demonstrate that fraction~(\ref{semic}) with $0<r<a_{n+1}$ violates~(\ref{rth below}) with $\ell=3$ for the choice $p'=p_{n-1}$, $q'=q_{n-1}$. That is, we shall verify inequality
\begin{equation}\label{k odd}
(q_{n}r+q_{n-1})^2\left(\alpha-\frac{p_{n}r+p_{n-1}}{q_{n}r+q_{n-1}}\right)\geq q_{n-1}^2\left(\alpha-\frac{p_{n-1}}{q_{n-1}}\right) 
\end{equation}
for every $r=1,\ldots,a_{n+1}-1$.
It is easy to transform~(\ref{k odd}) into 
\begin{equation*}\label{k odd r}
r\leq\frac{2q_{n}q_{n-1}\alpha-q_{n}p_{n-1}-q_{n-1}p_{n}}{q_{n}(p_{n}-q_{n}\alpha)} \qquad\mbox{for every $r=1,\ldots,a_{n+1}-1$},
\end{equation*}
which is further equivalent to
\begin{equation}\label{k odd a}
a_{n+1}-1\leq\frac{2q_{n} q_{n-1}\alpha-q_{n} p_{n-1}-q_{n-1}p_{n}}{q_{n}(p_{n}-q_{n}\alpha)}\;.
\end{equation}
From identity~(\ref{Kh(8)}) we obtain that the numerator on the right hand side of~(\ref{k odd a}) is equal to $2q_{n} q_{n-1}\alpha-2q_{n-1}p_{n}+1$. Therefore, (\ref{k odd a}) can be rewritten as
\begin{equation}\label{k odd a1}
a_{n+1}-1\leq-2\frac{q_{n-1}}{q_{n}}+\frac{1}{q_{n}(p_{n}-q_{n}\alpha)}\;.
\end{equation}
Now we express the right hand side of~(\ref{k odd a1}) in terms of $\alpha=[a_0;a_1,a_2,\ldots]$. Equations~(\ref{Kh Thm6}) and~(\ref{Ha.(2.2)}) together with the identity $[a_{n};a_{n-1},\ldots,a_1]^{-1}=[0;a_{n},a_{n-1},\ldots,a_1]$ yield that we can write the right hand side of~(\ref{k odd a1}) as
\begin{equation*}
-2[0;a_{n},a_{n-1},\ldots,a_1]+[a_{n+1};a_{n+2},\ldots]+[0;a_{n},a_{n-1},\ldots,a_1]\,.
\end{equation*}
Hence (\ref{k odd a1}) has the form
\begin{equation*}
a_{n+1}-1\leq-[0;a_{n},a_{n-1},\ldots,a_1]+a_{n+1}+[0;a_{n+2},a_{n+3},\ldots]\,.
\end{equation*}
This inequality can be simplified to
\begin{equation*}
[0;a_{n},a_{n-1},\ldots,a_1]\leq[1;a_{n+2},a_{n+3},\ldots]\,,
\end{equation*}
which is valid for any $\alpha=[a_0;a_1,a_2,\ldots]$.

(ii)\; We start again from Theorem~\ref{Thm.12} and Observation~\ref{Obs.l'<l}, which imply that each $\BUDA(3)$ to $\alpha$ has the form~(\ref{semic}) for some even $n\in\mathbb{N}_0$ and $0\leq r< a_{n+1}$. Our goal is to prove that the semiconvergents, which correspond to $0<r<a_{n+1}$, are either equal to $\frac{\lceil\alpha\rceil}{1}$ or violate the definition of $\BUDA(3)$. The proof falls into three cases: \{$n$ is even nonzero\}; \{$n=0$ and $r>1$\}; \{$n=0$ and $r=1$\}.

$\bullet$\; Let $n$ be even positive integer. We prove that each fraction~(\ref{semic}) with $0<r<a_{n+1}$ violates~(\ref{rth above}) with $\ell=3$ and $p'=p_{n-1}$, $q'=q_{n-1}$. Similarly as in part~(i), but this time for an even $n$, we verify the inequality
\begin{equation}\label{k even a}
(q_{n}r+q_{n-1})^2\left(\frac{p_{n}r+p_{n-1}}{q_{n}r+q_{n-1}}-\alpha\right)\geq q_{n-1}^2\left(\frac{p_{n-1}}{q_{n-1}}-\alpha\right)
\end{equation}
for every $r=1,\ldots,a_{n+1}-1$. We again transform~(\ref{k even a}) into
\begin{equation}\label{k even a1}
a_{n+1}-1\leq-2\frac{q_{n-1}}{q_{n}}+\frac{1}{q_{n}(q_{n}\alpha-p_{n})}
\end{equation}
and subsequently rewrite~(\ref{k even a1}) in the form
\begin{equation*}\label{k even ineq}
[0;a_{n},a_{n-1},\ldots,a_1]\leq[1;a_{n+2},a_{n+3},\ldots]\,,
\end{equation*}
which is valid for any $\alpha=[a_0;a_1,a_2,\ldots]$.

$\bullet$\; If $n=0$ and $r>1$, we will show that fraction~(\ref{semic}), i.e.,
\begin{equation*}
\frac{p_0 r+p_{-1}}{q_0 r+q_{-1}}=\frac{a_0 r+1}{1\cdot r+0}=\frac{a_0 r+1}{r}\,,
\end{equation*}
violates condition~(\ref{rth above}) with $\ell=3$ for the choice $p'=a_0+1$, $q'=1$. To prove this we shall verify inequality
\begin{equation*}\label{k0r}
r^{2}\left(\frac{a_0 r+1}{r}-\alpha\right)\geq1^{2}\left(\frac{a_0+1}{1}-\alpha\right),
\end{equation*}
which is equivalent to
\begin{equation}\label{k0r equiv}
(r-1)\left[(r+1)(a_0-\alpha)+1\right]\geq0\,.
\end{equation}
Since $1<r\leq a_1-1$ and $\alpha\leq a_0+\frac{1}{a_1}$, we have $(r+1)(a_0-\alpha)+1\geq a_1\cdot\frac{-1}{a_1}+1\geq0$. So~(\ref{k0r equiv}) holds.

$\bullet$\; Finally, consider (\ref{semic}) for $n=0$ and $r=1$, i.e.,
\begin{equation}\label{sem01}
\frac{p_0\cdot1+p_{-1}}{q_0\cdot1+q_{-1}}=\frac{a_0 r+1}{1\cdot 1+0}=\frac{a_0+1}{1}\,.
\end{equation}
If $n=0$, the case $r=1$ is possible only when $a_1>1$ (see~(\ref{semiconvergents})). Hence necessarily $\alpha=[a_0;a_1,\ldots]=a_0+\frac{1}{a_1+\cdots}\notin\mathbb{Z}$. In this case semiconvergent~(\ref{sem01}) is equal to $\frac{\lceil\alpha\rceil}{1}$.
\end{proof}

It is easy to check that the necessary condition from Theorem~\ref{Thm.3} is not sufficient.
We formulate a necessary and sufficient condition in Proposition~\ref{Prop.3} below.

\begin{proposition}\label{Prop.3}
Let $n$ be a positive integer and $\alpha=[a_0;a_1,a_2,\ldots]$. 
Then we have:
\begin{itemize}
\item[(i)] A convergent $\frac{p_n}{q_n}$ is the best lower approximation of the 3rd kind to $\alpha$ if and only if $n$ is even and 
\begin{equation}\label{cond.3}
\fl [a_{k+1};a_{k+2},\ldots]+[0;a_k,a_{k-1},\ldots,a_1]<[a_{n+1};a_{n+2},\ldots]+[0;a_n,a_{n-1},\ldots,a_1]
\end{equation}
holds for all $k=n-2,n-4,\ldots,2,0$.
\item[(ii)] A convergent $\frac{p_n}{q_n}$ is the best upper approximation of the 3rd kind to $\alpha$ if and only if $n$ is odd and (\ref{cond.3}) holds for all $k=n-2,n-4,\ldots,3,1$.
\end{itemize}
\end{proposition}

\begin{proof}
(i)\; From Theorem~\ref{Thm.3} we obtain that the only possible candidates for $\BLDA(3)$ to $\alpha$ are even--order convergents of $\alpha$. Therefore, setting $\mathcal{S}_L=\left\{\frac{p_n}{q_n}\,:\,\mbox{$n$ is even}\right\}$ in Proposition~\ref{Prop.crit}(i), we infer that $\frac{p_n}{q_n}$ for an even $n$ is a $\BLDA(3)$ if and only if
\begin{equation*}
q_n^2\left(\alpha-\frac{p_n}{q_n}\right)<q_k^2\left(\alpha-\frac{p_k}{q_k}\right) \qquad \mbox{for all even $k<n$}.
\end{equation*}
This and formula~(\ref{Ha.(2.2)}) imply 
$$
\frac{1}{[a_{n+1};a_{n+2},\ldots]+[0;a_n,a_{n-1},\ldots,a_1]}<\frac{1}{[a_{k+1};a_{k+2},\ldots]+[0;a_k,a_{k-1},\ldots,a_1]}
$$
for all even $k<n$, and criterion (i) follows immediately.

(ii)\; From Theorem~\ref{Thm.3} we get the set of candidates for $\BUDA(3)$ to $\alpha$ in the form $\mathcal{S}_U=\left\{\frac{p_n}{q_n}\,:\,\mbox{$n$ is odd}\right\}\cup\left\{\frac{\lceil\alpha\rceil}{1}\right\}$. Proposition~\ref{Prop.crit}(ii) then implies that $\frac{p_n}{q_n}$ with an odd $n$ is a $\BUDA(3)$ if and only if
\numparts
\begin{eqnarray}
q_n^2\left(\alpha-\frac{p_n}{q_n}\right)<q_k^2\left(\alpha-\frac{p_k}{q_k}\right) \quad \mbox{for all odd $k<n$} \label{BUDA3 conv} \\
\mbox{and}\quad
q_n^2\left(\frac{p_n}{q_n}-\alpha\right)<1^2\left(\frac{\lceil\alpha\rceil}{1}-\alpha\right). \label{BUDA3 1}
\end{eqnarray}
\endnumparts
Now we will show that (\ref{BUDA3 conv}) implies (\ref{BUDA3 1}). To prove this, we will demonstrate that
\begin{equation}\label{1ceil}
1^2\left(\frac{\lceil\alpha\rceil}{1}-\alpha\right)\geq q_1^2\left(\frac{p_1}{q_1}-\alpha\right).
\end{equation}
Since $\frac{p_1}{q_1}=\frac{a_0a_1+1}{a_1}$, we easily rewrite~(\ref{1ceil}) as
\begin{equation*}
a_0+1-\alpha\geq a_1(a_0a_1+1-a_1\alpha)\,,
\end{equation*}
which is equivalent to
\begin{equation}\label{1ceil2}
(a_1-1)\bigl[(a_1+1)(\alpha-a_0)-1\bigr]\geq0\,.
\end{equation}
In order to prove~(\ref{1ceil2}), we estimate
\begin{equation}\label{estim}
\alpha-a_0=\frac{1}{a_1+\frac{1}{a_2+\cdots}}>\frac{1}{a_1+1}\,,
\end{equation}
where the term $\frac{1}{a_2+\cdots}$ is smaller than $1$, because an expansion $\alpha=[a_0;a_1,1]$ with the last term $a_2=1$ is excluded, see Section~\ref{Sect.CF}.
With regard to~(\ref{estim}), inequality~(\ref{1ceil2}) is true, so (\ref{1ceil}) is verified.
We conclude that $\frac{p_n}{q_n}$ for an odd $n$ is a $\BUDA(3)$ if and only if (\ref{BUDA3 conv}) holds true. Finally, (\ref{BUDA3 conv}) corresponds to~(\ref{cond.3}) by virtue of~(\ref{Ha.(2.2)}); see part (i) of the proof.
\end{proof}

The following proposition will be used in a physical application in Section~\ref{Application}.
\begin{proposition}\label{Prop.Lebesgue}
Almost all $\alpha\in\mathbb{R}$ have infinitely many $\BLDA(3)$ and infinitely many $\BUDA(3)$. 
\end{proposition}

\begin{proof}
We prove that the set $\mathbb{M}=\{ \alpha ; \alpha \ \mbox{has finitely many }\BLDA(3)\}$ has zero Lebesgue measure.
Let $\alpha=[a_0;a_1,a_2,a_3,\ldots]\in\mathbb{M}\backslash\mathbb{Q}$ be fixed. For every even $n$, let us set
\begin{equation}\label{Pn}
P(n)=[a_{n+1};a_{n+2},\ldots]+[0;a_n,a_{n-1},\ldots,a_1]
\end{equation}
and define $H(n)=\max\{P(n-2),P(n-4),\ldots,P(2),P(0)\}$.
We have immediately that $H(n)\geq H(n-2)$ for every even $n\in\mathbb{N}$.

If $n$ has property $H(n)>H(n-2)$, then~(\ref{cond.3}) holds for all $k=n-2,n-4,\ldots,2,0$; thus $\frac{p_n}{q_n}$ is a $\BLDA(3)$ to $\alpha$ due to Proposition~\ref{Prop.3}.
Our assumption $\alpha\in\mathbb{M}$ implies that there are only finitely many such $n$. Therefore, the sequence $\{H(2n)\}_{n=1}^\infty$ is eventually constant.

Consequently, values $P(n)$ for even $n$ are bounded. From~(\ref{Pn}) we obtain that every $\alpha\in\mathbb{M}\backslash\mathbb{Q}$ has bounded terms at odd positions of its continued fraction expansion.
Hence $\mathbb{M}\subset \mathbb{Q}\cup\bigcup_{j=1}^\infty M_j$ where for each $j\in\mathbb N$ we have 
$M_j=\{ \alpha ;\alpha =[a_0;a_1,a_2,\dots ], a_{2k}<j\ \mbox{for every}\ k\in\mathbb N\}$. 
But Theorem 2.1 from \cite{Ha17}---see also Remark 2.1 and paragraph after Remark 2.1 of \cite{Ha17}---yields that $M_j$ has zero Lebesgue measure for every $j\in\mathbb N$. Hence the set $\mathbb M$ has zero Lebesgue measure. 

The proof that  almost all $\alpha\in\mathbb{R}$ have infinitely many $\BUDA(3)$  is similar. 
\end{proof}

\section{Approximations of the third kind for quadratic numbers}\label{Sect.quadratic}

The criterion derived in Proposition~\ref{Prop.3} is particularly convenient if the continued fraction of $\alpha$ has some regular structure. A prominent example are eventually periodic continued fractions,
\begin{equation}\label{quadratic alpha}
\alpha=[a_0;a_1,\ldots,a_m,\overline{a_{m+1},\ldots,a_{m+h}}]\,.
\end{equation}
Due to a classical result by Euler and Lagrange, periodic continued fractions correspond to quadratic irrational numbers, i.e., irrational roots of polynomials $x^2+ux+v$ with $u,v\in\mathbb{Q}$.

In this section, we apply Proposition~\ref{Prop.3} on a general quadratic irrational number $\alpha$ to find bounds on the number of its best upper and lower approximations of the third kind. In particular, we show that the set of $\BLDA(3)$ and the set of $\BUDA(3)$ cannot be both infinite.

\begin{theorem}\label{Thm.quadratic}
Let $\alpha$ be given as~(\ref{quadratic alpha})
for some non-negative integer $m$ and a positive integer $h$.

(i)\; If $m=0$; or $m$ is odd and $a_m<a_{m+h}$; or $m$ is even nonzero and $a_m>a_{m+h}$, then the number of best upper approximations of the 3rd kind to $\alpha$ is finite.

(ii)\; If $m$ is odd and $a_m>a_{m+h}$; or $m$ is even nonzero and $a_m<a_{m+h}$, then the number of best lower approximations of the 3rd kind to $\alpha$ is finite.

(iii)\; A quadratic irrational number cannot have infinitely many $\BLDA(3)$ and infinitely many $\BUDA(3)$ at the same time.
\end{theorem}

\begin{proof}
(i)\; Due to Theorem~\ref{Thm.3}, each $\BUDA(3)$ to $\alpha$ is either $\frac{\lceil\alpha\rceil}{1}$ or an odd--order convergent of $\alpha$. We will show that for any odd $n>m+2h$, $\frac{p_n}{q_n}$ is not a $\BUDA(3)$ to $\alpha$.

Let us thus consider an arbitrary odd $n>m+2h$. According to Proposition~\ref{Prop.3}, $\frac{p_n}{q_n}$ is a $\BUDA(3)$ only if~(\ref{cond.3}) holds for every odd $k<n$. We take in particular $k=n-2h$ (one can take also $k=n-h$ if $h$ is even) and rewrite~(\ref{cond.3}) in terms of $(n,k)=(k+2h,k)$. We obtain
\begin{equation}\label{condition quadratic}
\fl [a_{k+1};a_{k+2},\ldots]+[0;a_k,a_{k-1},\ldots,a_1]<[a_{k+2h+1};a_{k+2h+2},\ldots]+[0;a_{k+2h},a_{k+2h-1},\ldots,a_1]\,.
\end{equation}
Since $k>m$ (recall that $n>m+2h$), we use the periodicity of representation~(\ref{quadratic alpha}) to conclude that $[a_{k+1};a_{k+2},\ldots]=[a_{k+2h+1};a_{k+2h+2},\ldots]$. This allows us to simplify condition~(\ref{condition quadratic}) to
\begin{equation}\label{cond. quadr.}
[0;a_k,a_{k-1},\ldots,a_1]<[0;a_{k+2h},a_{k+2h-1},\ldots,a_1]\,.
\end{equation}
Now we shall demonstrate that (\ref{cond. quadr.}) is violated in all the three cases from statement (i), i.e., \{$m=0$\}; \{$m$ is odd and $a_m<a_{m+h}$\}; \{$m$ is even nonzero and $a_m>a_{m+h}$\}.

$\bullet$\; If $m=0$, we have $[0;a_{k+2h},a_{k+2h-1},\ldots,a_1]=[0;a_{k},a_{k-1},\ldots,a_1,a_{h},\ldots,a_1,a_{h},\ldots,a_1]$. Then condition~(\ref{cond. quadr.}) takes the form
\begin{equation}\label{m0}
[0;a_k,a_{k-1},\ldots,a_1]<[0;a_{k},a_{k-1},\ldots,a_1,a_{h},\ldots,a_1,a_{h},\ldots,a_1]\,.
\end{equation}
Since $k$ is odd, inequality (\ref{m0}) is false in view of Proposition~\ref{Prop.comparison}(ii). Thus $\frac{p_n}{q_n}$ is not a $\BUDA(3)$ to $\alpha$.

$\bullet$\; Let $m>0$. Then we have
\begin{equation*}
[0;a_k,a_{k-1},\ldots,a_1]=[0;a_k,a_{k-1},\ldots,a_{m+1},a_m,a_{m-1},\ldots,a_1]
\end{equation*}
and
$$
[0;a_{k+2h},a_{k+2h-1},\ldots,a_1]=[0;a_{k},a_{k-1},\ldots,a_{m+1},a_{m+h},\ldots,a_{m+1},a_{m+h},\ldots,a_{m+1},a_m,\ldots,a_1]\,.
$$
Hence (\ref{cond. quadr.}) has the form
\begin{equation}\label{m>0}
\fl \eqalign{
[0;a_k,a_{k-1},\ldots,a_{m+1},a_m,a_{m-1},\ldots,a_1] \\
\qquad<[0;a_{k},a_{k-1},\ldots,a_{m+1},a_{m+h},\ldots,a_{m+1},a_{m+h},\ldots,a_{m+1},a_m,\ldots,a_1]\,.}
\end{equation}
Now if $m$ is odd and $a_m<a_{m+h}$, we have that $k-m+1$ is odd and $a_m<a_{m+h}$, thus (\ref{m>0}) is false by Proposition~\ref{Prop.comparison}(i). Similarly, if $m$ is even and $a_m>a_{m+h}$, we have that $k-m+1$ is even and $a_m>a_{m+h}$, so (\ref{m>0}) is again false. Therefore, in either case $\frac{p_n}{q_n}$ is not a $\BUDA(3)$ to $\alpha$.

(ii)\; The proof is similar to (i), with the main difference that we examine even $n>m-2h$, thus $k=n-2h$ is even. One proves that no convergent $\frac{p_n}{q_n}$ with $n>m+2h$ is a $\BLDA(3)$ to $\alpha$.

(iii)\; A quadratic irrational number has an eventually periodic continued fraction of form~(\ref{quadratic alpha}), so statements (i) and (ii) apply. The conditions listed in (i) and (ii) are complementary. As one of them is always satisfied, either the number of $\BLDA(3)$ to $\alpha$ or the number of $\BUDA(3)$ to $\alpha$ must be finite.
\end{proof}

\begin{proposition}\label{Prop.bound}
Let $\alpha=[a_0;a_1,\ldots,a_m,\overline{a_{m+1},\ldots,a_{m+h}}]$ for $m\in\mathbb N_0$ and $h\in\mathbb N$.

(i)\; If $\alpha$ has infinitely many $\BLDA(3)$, then $\alpha$ has at most $(1+\lceil m/2\rceil+h)$ $\BUDA(3)$.

(ii)\; If $\alpha$ has infinitely many $\BUDA(3)$, then $\alpha$ has at most $(\lfloor m/2\rfloor+h)$ $\BLDA(3)$.
\end{proposition}

\begin{proof}
(i)\; We will apply Theorem~\ref{Thm.quadratic}. The case of infinitely many $\BLDA(3)$ to $\alpha$ corresponds to case (i) of Theorem~\ref{Thm.quadratic}. The proof of Theorem~\ref{Thm.quadratic}(i) then implies that every $\BUDA(3)$ to $\alpha$ is either $\frac{\lceil\alpha\rceil}{1}$ or a convergent $\frac{p_n}{q_n}$ of $\alpha$ for an odd $n\leq m+2h$. In total there are at most $1+\left\lceil\frac{m+2h}{2}\right\rceil$ possibilities.

(ii)\; Infinitely many $\BUDA(3)$ to $\alpha$ correspond to case (ii) of Theorem~\ref{Thm.quadratic}. So each $\BLDA(3)$ to $\alpha$ must be a convergent $\frac{p_n}{q_n}$ of $\alpha$ for an even $n\leq m+2h$. Hence we get at most $\left\lfloor\frac{m+2h}{2}\right\rfloor$ possibilities.
\end{proof}

\begin{remark}
The bounds on the number of $\BLDA(3)$ and $\BUDA(3)$ to $\alpha$ given in Proposition~\ref{Prop.bound} can be improved, but we will not go into detail for the sake of simplicity of the proof.
\end{remark}

\section{Approximations of the $\ell$-th kind for $\ell\geq4$}\label{Section:4th}

Theorem~\ref{Thm.3} together with Observation~\ref{Obs.l'<l} imply that every best lower or upper approximation of the $\ell$-th kind to $\alpha$ for $\ell\geq4$ is either a convergent of $\alpha$ or $\frac{\lceil\alpha\rceil}{1}$. Note at first that the sets of $\BLDA(\ell)$ and $\BUDA(\ell)$ to $\alpha$ are always nonempty:

\begin{observation}\label{Obs.4first}
For every $\ell\in\mathbb{N}$ and $\alpha\in\mathbb{R}$, $\frac{p_0}{q_0}=\frac{\lfloor\alpha\rfloor}{1}$ is a $\BLDA(\ell)$ to $\alpha$ and $\frac{\lceil\alpha\rceil}{1}$ is a $\BUDA(\ell)$ to $\alpha$.
\end{observation}

However, as $\ell$ grows beyond $3$, the structure of the sets of $\BLDA(\ell)$ and $\BUDA(\ell)$ to a given $\alpha=[a_0;a_1,a_2,a_3,\ldots]$ becomes increasingly dependent on the values of $a_j$. Consider the following proposition:

\begin{proposition}\label{Prop.crit.4}
Let $\ell\geq4$. For a given $\alpha=[a_0;a_1,a_2,a_3,\ldots]$, set
\begin{equation}\label{Cnl}
C_n(\ell)=\frac{q_n^{\ell-3}}{[a_{n+1};a_{n+2},\ldots]+[0;a_n,a_{n-1},\ldots,a_1]}\;,
\end{equation}
where $\frac{p_n}{q_n}$ is the $n$-th convergent of $\alpha$.
Then
$\frac{p_n}{q_n}$ is a best lower approximation of the $\ell$-th kind to $\alpha$ if and only if $n$ is even and $C_n(\ell)<C_k(\ell)$ for all $k=0,2,4,6,\ldots,n-2$.
\end{proposition}

\begin{proof}
The proof is similar to the proof of Proposition~\ref{Prop.3}(i).
We use Theorem~\ref{Thm.3} together with Observation~\ref{Obs.l'<l} to infer that every $\BLDA(\ell)$ to $\alpha$ is an even--order convergent of $\alpha$. Then we apply Proposition~\ref{Prop.crit}(i) with $\mathcal{S}_L=\left\{\frac{p_n}{q_n}\,:\,\mbox{$n$ is even}\right\}$, whence we obtain that $\frac{p_n}{q_n}$ for an even $n$ is a $\BLDA(\ell)$ if and only if
\begin{equation}\label{cond.4}
q_n^{\ell-1}\left(\alpha-\frac{p_n}{q_n}\right)<q_k^{\ell-1}\left(\alpha-\frac{p_k}{q_k}\right) \qquad \mbox{for all even $k<n$}.
\end{equation}
Finally, we use~(\ref{Ha.(2.2)}) to rewrite condition~(\ref{cond.4}) in the form
$$
\frac{q_n^{\ell-3}}{[a_{n+1};a_{n+2},\ldots]+[0;a_n,a_{n-1},\ldots,a_1]}<\frac{q_k^{\ell-3}}{[a_{k+1};a_{k+2},\ldots]+[0;a_k,a_{k-1},\ldots,a_1]}
$$
for all even $k<n$.
\end{proof}

Let us comment on Proposition~\ref{Prop.crit.4}. Recall that $q_n$ depends solely on terms $a_j$ for $j\leq n$; cf.~(\ref{rekurence}). So does the numerator of $C_n(\ell)$ in expression~(\ref{Cnl}), while the denominator has $a_{n+1}$ as its dominant term. Therefore, $\frac{p_n}{q_n}$ is a best lower approximation to $\alpha$ if and only if $a_{n+1}$ is large enough compared to the quantity $q_n\in\left[\prod_{j=1}^n a_j,\prod_{j=1}^n(a_j+1)\right)$. Hence we conclude that the number of $\BLDA(\ell)$ to a given $\alpha$ can in general attain any value from $1$ to infinity depending on the arrangement of large terms at odd positions in the continued fraction expansion of $\alpha$.
Similar results can be derived for best upper approximations of the $\ell$-th kind.

In particular, since the numerators of $C_n(\ell)$ in~(\ref{Cnl}) grow to infinity (if $\ell\geq4$), Proposition~\ref{Prop.crit.4} and the considerations above have a straightforward consequence:

\begin{observation}
If $\ell\geq4$ and the terms $a_n$ with odd indices $n$ in $\alpha=[a_0;a_1,a_2,a_3,\ldots]$ are bounded, then $\alpha$ has only finitely many best lower approximations of the $\ell$-th kind. Similarly, if the terms $a_n$ with even indices $n$ are bounded, there are only finitely many $\BUDA(\ell)$.
\end{observation}

We can even say more:

\begin{proposition}\label{A}
(i)\; Let $\ell$ be a positive integer such that $\ell\geq 4$ and let $\{ a_n\}_{n=1}^\infty$ be a  sequence of positive integers such that 
\begin{equation}\label{goldenratio}
\limsup_{n\to\infty} \frac {\log a_{2n+1}}{2n+1} <(\ell-3)\log \varphi\,,
\end{equation}
where $\varphi=\frac{1+\sqrt{5}}{2}$ is the golden ratio. Then the number $\alpha=[a_0;a_1,a_2,a_3,\ldots]$  has only finitely many best lower approximations of the $\ell$-th kind.

(ii)\; Similarly, if we have
\begin{equation}\label{goldenratioB}
\limsup_{n\to\infty} \frac {\log a_{2n}}{2n} <(\ell-3)\log \varphi\,,
\end{equation}
then $\alpha=[a_0;a_1,a_2,a_3,\ldots]$ has only finitely many best upper approximations of the $\ell$-th kind.
\end{proposition}
\begin{proof} 
We will prove the part (i); the proof of (ii) is similar. In view of Proposition~\ref{Prop.crit.4}, let us examine the quantity $C_n(\ell)$ for even numbers $n$. First of all, we have trivially
\begin{equation}\label{Cnl2}
C_n(\ell)\geq\frac{q_{n}^{\ell -3}}{a_{n+1}+\frac{1}{a_{n+2}}+\frac{1}{a_{n}}}\geq  \frac{q_{n}^{\ell -3}}{a_{n+1}+2}\,.
\end{equation}
Now we will estimate the numerator and denominator of~(\ref{Cnl2}).
From (\ref{goldenratio}) we obtain that there exists an $x$ and a $k_0$ such that $1< x<\varphi$ and $\log a_{k} <(\ell-3)k\log x$ for all odd $k>k_0$. Taking in particular the odd integer $k=n+1$ (recall that $n$ is even), we have
\begin{equation}\label{estimation}
a_{n+1}<x^{(\ell-3)(n+1)} \qquad\mbox{for all even $n\geq k_0$.}
\end{equation}
Using the recurrent relation~(\ref{rekurence}), we get $q_n\geq F_n$ for all $n$, where $F_n=\frac 1{\sqrt 5}(\varphi^n-\varphi^{-n})$ is the $n$-th Fibonacci number; note that the equality $q_n=F_n$ holds iff $1=a_1=\dots =a_n$.
When we plug the estimate $q_n\geq F_n$ and~(\ref{estimation}) into~(\ref{Cnl2}), we get
\begin{equation}\label{fibonacci}
C_n(\ell)\geq\frac{1}{(\sqrt{5})^{\ell-3}}\cdot
\frac{\left(\varphi^{n}-\varphi^{-n}\right)^{\ell -3}}{x^{(\ell-3)(n+1)}+2} \qquad\mbox{for all even $n\geq k_0$.}
\end{equation}
Now inequality $1< x<\varphi$  yields 
\begin{equation*}
\lim_{n\to\infty} \frac{\left(\varphi^{n}-\varphi^{-n}\right)^{\ell -3}}{x^{(\ell-3)(n+1)}+2}=\infty.
\end{equation*}
As a particular consequence of this and (\ref{fibonacci}),
there exists $n_0$ such that for all even $n>n_0$, we have $C_n(\ell)\nless C_0(\ell)$.
Then Proposition~\ref{Prop.crit.4} implies that a convergent $\frac{p_n}{q_n}$ is a best lower approximations of the $\ell$-th kind to $\alpha$ only if $n\leq n_0$. Consequently, the number of $\BLDA(\ell)$ to $\alpha$ is finite.
\end{proof}

Both sets of $\BLDA(\ell)$ and $\BUDA(\ell)$ are finite also in the case when $\alpha$ is an irrational algebraic number, i.e., an irrational root of a polynomial with integer coefficients:

\begin{proposition}\label{algebraic}
For all $\ell\geq4$, every irrational algebraic number $\alpha$ has a finite number of best upper and best lower approximations of the $\ell$-th kind.
\end{proposition}

\begin{proof}
Let us prove that the number of $\BLDA(\ell)$ is finite. The case of $\BUDA(\ell)$ is similar.
Let $\alpha$ be an irrational algebraic number. Roth's theorem states that for each $\varepsilon>0$ there are finitely many coprime integers $p,q$ such that
\begin{equation*}
\left|\alpha-\frac{p}{q}\right|<\frac{1}{q^{2+\varepsilon}}\,.
\end{equation*}
Setting in particular $\varepsilon=\ell-3$, we obtain that for any $\ell>3$ there exist only finitely many integers $p,q$ such that
\begin{equation}\label{pqRoth}
q^{\ell-1}\left|\alpha-\frac{p}{q}\right|<1.
\end{equation}
At the same time, the choice $p'=\lfloor\alpha\rfloor$, $q'=1$ gives
\begin{equation}\label{p'q'}
(q')^{\ell-1}\left(\alpha-\frac{p'}{q'}\right)=1^{\ell-1}\left(\alpha-\frac{\lfloor\alpha\rfloor}{1}\right)=\alpha-\lfloor\alpha\rfloor<1\,.
\end{equation}
From~(\ref{pqRoth}) and~(\ref{p'q'}), we obtain that there are only finitely many rational numbers $\frac{p}{q}<\alpha$ such that
\begin{equation*}
0\leq q^{\ell-1}\left(\alpha-\frac{p}{q}\right)<1^{\ell-1}\left(\alpha-\frac{\lfloor\alpha\rfloor}{1}\right).
\end{equation*}
In other words, only finitely many rational numbers $\frac{p}{q}$ can satisfy the definition of a best lower approximation to $\alpha$ (Definition~\ref{Def. rth}).
\end{proof}

Let us conclude this section with describing metric properties of the sets of numbers having infinitely many best one--sided approximations of the $\ell$-th kind.
\begin{proposition}
For every $\ell\geq4$ the set of numbers $\alpha$ which have infinitely many best upper or lower approximations of the $\ell$-th kind has zero Lebesgue measure. 
\end{proposition}
\begin{proof}
To obtain the statement, we use the fact that for every positive real $\varepsilon$ the set 
\begin{equation}\label{SS}
\fl \mathbb S=\left\{ \alpha \; ; \; \left| \alpha -\frac pq\right| <\frac 1{q^{2+\varepsilon}}\; \mbox{has infinitely many solutions}\; (p,q)\in(\mathbb Z,\mathbb Z)\right\}
\end{equation}
has zero Lebesgue measure \cite[p.~103]{bug}. Then we put $\varepsilon=\ell-3$ and follow the steps in the proof of Proposition \ref{algebraic}.
\end{proof}

\begin{proposition}
For every $\ell\geq4$ the set of numbers $\alpha$ which have infinitely many best upper or lower approximations of the $\ell$-th kind has Hausdorff dimension at most $\frac 2{\ell-1}$.
\end{proposition}
\begin{proof}
We again proceed similarly as in the proof of Proposition \ref{algebraic}, starting from the fact that for every positive real $\varepsilon$ the set $\mathbb S$, given by~(\ref{SS}),
has Hausdorff dimension $\frac 2{2+\varepsilon}$, which can be found in \cite[p.~104]{bug}. 
\end{proof}

\section{Application in mathematical physics}\label{Application}

We have seen in Remark~\ref{Geometric} that best one--sided approximations of the 2nd kind have a simple geometric interpretation. In this section we will present an advanced application of best one--sided Diophantine approximations of the 3rd kind by demonstrating their use in quantum mechanics on graphs.

The motivation for the problem arises in spectral analysis. When studying a quantum system consisting of a particle confined to an infinite periodic rectangular network with $\delta$-type potentials in the vertices (see Figure~\ref{Lattice}),
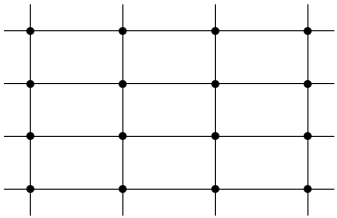
\begin{figure}[h]
\setlength{\unitlength}{1pt}
\begin{center}
\begin{picture}(125,80)
\multiput(10,0)(35,0){4}{\line(0,1){80}}
\multiput(0,10)(0,20){4}{\line(1,0){125}}
\multiput(10,10)(35,0){4}{\circle*{3}}
\multiput(10,30)(35,0){4}{\circle*{3}}
\multiput(10,50)(35,0){4}{\circle*{3}}
\multiput(10,70)(35,0){4}{\circle*{3}}
\end{picture}
\end{center}
\caption{A periodic rectangular lattice graph with $\delta$ potentials (represented by solid circles) in the vertices. A particle is confined to the edges of the graph.}
\label{Lattice}
\end{figure}
one finds that the system has gaps in its energy spectrum. In other words, there are intervals of energies that the particle cannot attain.
If we denote the lengths of the edges of the rectangle by $a$ and $b$ and consider a repulsive $\delta$ potential of strength $u>0$, it can be proved that every gap is adjacent to some of the points $(m\pi/a)^2$ and $(m\pi/b)^2$, where $m\in\mathbb{N}$ is a positive integer~\cite{Ex95}. The presence or absence of a gap at a given position $(m\pi/a)^2$ or $(m\pi/b)^2$ depends on the parameter $u$. A calculation shows~\cite{Ex96} that a gap adjacent to $(m\pi/a)^2$ is present if and only if the integer $m\in\mathbb{N}$ satisfies
\begin{equation}\label{cond a}
\frac{2m}{\pi}\tan\left(\frac{\pi}{2}(m\left(\frac{b}{a}-\left\lfloor m\frac{b}{a}\right\rfloor\right)\right)<\frac{u a}{\pi^2}\,.
\end{equation}
Similarly, a gap adjacent to $(m\pi/b)^2$ is present if and only if
\begin{equation}\label{cond b}
\frac{2m}{\pi}\tan\left(\frac{\pi}{2}(m\left(\frac{a}{b}-\left\lfloor m\frac{a}{b}\right\rfloor\right)\right)<\frac{u b}{\pi^2}\,.
\end{equation}
Conditions~(\ref{cond a}) and (\ref{cond b}) have a slightly different form in the case of attractive $\delta$ potentials (see~\cite{ET17}), but we will not go into details here.

We will demonstrate in Theorem~\ref{Thm.number} below that a certain information about the set of best lower Diophantine approximations of the third kind to $b/a$ and to $a/b$ allows to formulate general statements regarding the gaps in the energy spectrum of the system. For proving the theorem, we will need the following lemma.

\begin{lemma}\label{Lemma monot}
Let $\{y_n\}_{n=1}^\infty$ be a strictly increasing sequence of positive numbers and $x_n\in\left[0,\frac{\pi}{2}\right)$ for all $n\in\mathbb{N}$.
If the sequence $\{y_n x_n\}_{n=1}^\infty$ is strictly decreasing, then the sequence $\{y_n\tan x_n\}_{n=1}^\infty$ is strictly decreasing.
\end{lemma}

\begin{proof}
The assumptions on $\{y_n x_n\}_{n=1}^\infty$ and $\{y_n\}_{n=1}^\infty$ give
\begin{equation}\label{decr x}
x_{n+1}<\frac{y_n x_n}{y_{n+1}}<x_n \qquad\mbox{for all $n\in\mathbb{N}$};
\end{equation}
thus the sequence $\{x_n\}_{n=1}^\infty$ is strictly decreasing as well.

Since $\tan0=0$ and tangent is a strictly convex function on the interval $[0,\pi/2)$, we have
\begin{equation}\label{tg}
x'<x \;\Rightarrow\; \tan x'<\frac{x'}{x}\tan x \qquad \mbox{for all $x,x'\in\left[0;\frac{\pi}{2}\right)$}\,.
\end{equation}
A particular choice $x=x_n$ and $x'=x_{n+1}$ in~(\ref{tg}) together with~(\ref{decr x}) gives
\begin{equation*}
\tan x_{n+1}<\frac{x_{n+1}}{x_n}\tan x_n<\frac{y_n}{y_{n+1}}\tan x_n\,.
\end{equation*}
Hence we obtain $y_{n+1}\tan x_{n+1}<y_n\tan x_n$ for all $n\in\mathbb{N}$.
\end{proof}

\begin{theorem}\label{Thm.number}
Let $a,b>0$. If both $a/b$ and $b/a$ have infinitely many best lower approximations of the 3rd kind, then the number of gaps in the energy spectrum of a periodic rectangular lattice quantum graph with repulsive $\delta$ potentials in the vertices and edge lengths $a$ and $b$ is either infinite or zero.
\end{theorem}

\begin{proof}
We have to analyze the number of integers $m\in\mathbb{N}$ that satisfy condition~(\ref{cond a}) or condition~(\ref{cond b}). At first we will examine~(\ref{cond a}).

Let $\theta=b/a$ and $\left\{\frac{k_n}{m_n}\,:\,n\in\mathbb{N}_0\right\}$ be the set of all $\BLDA(3)$ to $\theta$. By assumption, this set has infinitely many elements. Observation~\ref{floorceil} gives $k_n=\lfloor m_n\theta\rfloor$. Without loss of generality, we can assume that the denominators form an increasing sequence, $m_0<m_1<m_2<\cdots$. Then the sequence $\left\{m_n(m_n\theta-\lfloor m_n\theta\rfloor)\right\}_{n=1}^\infty$ is strictly decreasing by Definition~\ref{Def. rth}. Moreover, the sequence has nonnegative terms; therefore
\begin{equation}\label{lim L}
\lim_{n\to\infty}m_n(m_n\theta-\lfloor m_n\theta\rfloor)=L\in[0,\infty)\,.
\end{equation}
As a particular consequence of~(\ref{lim L}), we have
\begin{equation}\label{lim0}
\lim_{n\to\infty}(m_n\theta-\lfloor m_n\theta\rfloor)=0\,.
\end{equation}
If we set in Lemma~\ref{Lemma monot}
\begin{equation*}
x_n=\frac{\pi}{2}(m_n\left(\theta-\left\lfloor m_n\theta\right\rfloor\right) \qquad\mbox{and}\qquad y_n=\frac{2m_n}{\pi}\,,
\end{equation*}
we obtain that the sequence
\begin{equation*}\label{seq tan}
\left\{\frac{2m_n}{\pi}\tan\left(\frac{\pi}{2}(m_n\theta-\left\lfloor m_n\theta\right\rfloor)\right)\right\}_{n=1}^\infty
\end{equation*}
is strictly decreasing.
Using~(\ref{lim L}) and~(\ref{lim0}), we find that
\begin{equation}\label{lim LS}
\eqalign{
\lim_{n\to\infty}\frac{2m_n}{\pi}\tan\left(\frac{\pi}{2}(m_n\theta-\left\lfloor m_n\theta\right\rfloor)\right) \\
=
\lim_{n\to\infty}m_n(m_n\theta-\lfloor m_n\theta\rfloor)\cdot\frac{\tan\left(\frac{\pi}{2}(m_n\theta-\left\lfloor m_n\theta\right\rfloor)\right)}{\frac{\pi}{2}(m_n\theta-\left\lfloor m_n\theta\right\rfloor)}
=L\cdot1=L\,.
}
\end{equation}
Now we are ready to analyze the number of integers $m\in\mathbb{N}$ that satisfy~(\ref{cond a}). We have two cases.

1. If $\frac{u a}{\pi^2}>L$, then (\ref{lim LS}) implies the existence of an $n_0$ such that
\begin{equation}\label{cond a lim}
\frac{2m_n}{\pi}\tan\left(\frac{\pi}{2}(m_n\theta-\left\lfloor m_n\theta\right\rfloor)\right)<\frac{u a}{\pi^2}
\end{equation}
for all $n>n_0$. Consequently, there are infinitely many integers $m_n$ satisfying~(\ref{cond a}).

2. Assume that $\frac{u a}{\pi^2}\leq L$. Then for every $m\in\mathbb{N}$, we have
\begin{equation}\label{LS odhad i}
\frac{2m}{\pi}\tan\left(\frac{\pi}{2}(m\theta-\left\lfloor m\theta\right\rfloor)\right)\geq
\frac{2m}{\pi}\cdot\frac{\pi}{2}(m\theta-\left\lfloor m\theta\right\rfloor)=
m(m\theta-\left\lfloor m\theta\right\rfloor)\,.
\end{equation}
If we take an arbitrary $\BLDA(3)$ of the form $\left\lfloor m_n\theta\right\rfloor/m_n$ with property $m_n\geq m$, then Definition~\ref{Def. rth} gives
\begin{equation}\label{LS odhad ii}
m_n(m_n\theta-\left\lfloor m_n\theta\right\rfloor)\leq m(m\theta-\left\lfloor m\theta\right\rfloor)\,.
\end{equation}
From~(\ref{LS odhad i}), (\ref{LS odhad ii}) and from the fact that sequence $\left\{m_n(m_n\theta-\lfloor m_n\theta\rfloor)\right\}_{n=1}^\infty$ strictly decreases to $L$ we get
\begin{equation*}
\frac{2m}{\pi}\tan\left(\frac{\pi}{2}(m\theta-\left\lfloor m\theta\right\rfloor)\right)>L \qquad \mbox{for all $m\in\mathbb{N}$}.
\end{equation*}
In other words, there exists no $m\in\mathbb{N}$ obeying~(\ref{cond a}).

In the same way one would analyze condition~(\ref{cond b}). The assumption that $a/b$ has infinitely many $\BLDA(3)$ leads to the conclusion that the number of solutions of~(\ref{cond b}) is either infinite or zero.

To sum up, the number of integers $m$ that satisfy at least one of the conditions~(\ref{cond a}), (\ref{cond b}) is either infinite or zero.
\end{proof}

This result is closely related to the existence of so-called Bethe--Sommerfeld quantum graphs. Let us finish this section with an important comment on this interesting problem.

The Bethe--Sommerfeld conjecture of 1933~\cite{BS33} states that any quantum system that is periodic in two or more directions has finitely many gaps in its energy spectrum. The conjecture was proved for several classes of systems (see e.g.~\cite{Sk79}), but turned out to be invalid for quantum graphs~\cite{BK13}. All examples of periodic quantum graphs studied in the literature until 2017 led to energy spectra with either infinitely many gaps, or no gaps at all. The first examples of quantum graphs that obey the conjecture in a nontrivial manner, i.e., that have a \emph{finite nonzero} number of gaps in their energy spectra, appeared in~\cite{ET17} and \cite{Tu18}. In accord with~\cite{ET17} let us call a quantum graph having a finite nonzero number of gaps in its energy spectrum to be of the \emph{Bethe--Sommerfeld type}. In view of Theorem~\ref{Thm.number}, we conclude that if the ratios of edge lengths $b/a$ and $a/b$ have infinitely many best lower approximations, then the periodic rectangular lattice graph in question cannot be of the Bethe--Sommerfeld type, regardless of the strength of the repulsive $\delta$ potential in the vertices.

Now let us recall Proposition~\ref{Prop.Lebesgue}, which says that the set of numbers $\alpha$ having infinitely many $\BLDA(3)$ has full Lebesgue measure. Hence we obtain immediately that the set of numbers $\alpha$ such that both $\alpha$ and $1/\alpha$ have infinitely many $\BLDA(3)$ has full Lebesgue measure as well. When we restrict our attention to the family of periodic rectangular graphs with repulsive $\delta$-type potentials in the vertices, we can say in view of Theorem~\ref{Thm.number} that the Bethe--Sommerfeld graphs form a subset of zero Lebesgue measure. For almost all ratios $a/b$ of edge lengths, the quantum graph in question does not belong to the Bethe--Sommerfeld class. This explains why it was so difficult and longstanding problem to prove the existence of Bethe--Sommerfeld graphs and, in particular, to find an explicit example.

\section{Conclusions}\label{Conclusions}

Let us compare the theory of the best one--sided (lower or upper) Diophantine approximations of the $\ell$-th kind ($\ell\in\mathbb{N}$)  with the theory of the classical best Diophantine approximation of the $\ell$-th kind. They have several differences and also some common features. The common property is that the both theories make use of convergents and semiconvergents as a main tool. Also metric properties are very similar. On the other hand, the structure of the sets of best lower and upper Diophantine approximations differs from the sets known in the classical theory. A surprising result was found for approximations of the first and second kind, which form mutually different sets in the classical theory, but in the theory of one--sided approximations they coincide (Theorem~\ref{Thm.12}).

An important aspect concerns applications and history. Classical ``double--sided'' best approximations have been developed and widely used in practical problems for centuries. Best lower and upper approximations, by contrast, do not have many known applications so far. In this paper, we demonstrated their immediate connection to quantum mechanics on graphs (Section~\ref{Application}), which originally served as a main motivation for our research. Our results help to understand the intricacy of Bethe--Sommerfeld graphs, the existence of which posed an open problem in mathematical physics for decades. We are certain that other applications of best one--sided approximations in physics and mathematics will arise in the future.

The research opens many interesting new questions. For instance, what will be the analog of the Lagrange or Markoff sequences? Will it be possible to obtain a one--sided version of Markoff chains? And if one constructs an analog of functions which substitute  Lagrange numbers and which are described in \cite{Ha15} or \cite{Ha16}, what form will they have? All of this could give rise to a nice theory.

\ack
The authors thank \v{S}t\v{e}p\'an Starosta (Czech Technical University in Prague) for valuable comments and discussions. The research was supported by the Czech Science Foundation (GA\v{C}R) within the project 17-01706S.


\section*{References}
\begin{thebibliography}{00}

\bibitem{BK13}
G. Berkolaiko, P. Kuchment: {\it Introduction to Quantum Graphs},
Amer. Math. Soc., Providence, R.I., 2013.

\bibitem{Be17} S.~Bettin: A congruence sum and rational approximations, {\it Rend.\ Circ.\ Mat.\ Palermo (2)} {\bf 66} (2017), 477--483.

\bibitem{bug}
Y. Bugeaud: Approximation by algebraic numbers, Cambridge Tracts in Mathematics 160, Cambridge University Press, 2004. 

\bibitem{Cu89}
T.W. Cusick, M. E. Flahive:  {\it The Markoff and Lagrange spectra}, Mathematical surveys and Monographs 30, American Mathematical Society, Providence, RI, 1989. 

\bibitem{EN61}
L.C.~Eggan, I.~Niven: A remark on one--sided approximation, {\it Proc. Amer. Math. Soc.} {\bf 12} (1961), 538--540.

\bibitem{Ex95}
P.~Exner: Lattice Kronig--Penney models, {\it Phys. Rev. Lett.} {\bf 74} (1995), 3503--3506.

\bibitem{Ex96}
P.~Exner: Contact interactions on graph superlattices, {\it J. Phys. A: Math. Gen.} {\bf 29} (1996), 87--102.

\bibitem{ET17}
P.~Exner, O.~Turek: Periodic quantum graphs from the Bethe--Sommerfeld perspective, {\it J.\ Phys.\ A: Math.\ Theor.} {\bf 50} (2017), 455201.

\bibitem{Fi93}
Y.Y.~Finkelshtein: Klein polygons and reduced regular continued fractions, {\it Russ. Math. Surv.} {\bf 48} (1993), 198--200.

\bibitem{Ha15} 
J.~Han\v{c}l: Sharpening of theorems of Vahlen and Hurwitz and approximation properties of the golden ratio, {\it Arch. Math.} (Basel) 105, no. 2, (2015), 129--137.  

\bibitem{Ha16}
J.~Han\v{c}l: Second basic theorem of Hurwitz, {\it Lith.\ Math.\ J.} {\bf 56} (2016), 72--76.

\bibitem{Ha17}
J.~Han\v{c}l, A. Ja\v s\v sov\' a and J. \v Sustek: Lebesgue measure and Hausdorff dimension of special sets of real numbers from $(0,1)$, {\it Ramanujan\ J.} {\bf 28} (2012), 15--23. 

\bibitem{Kh64}
A.Ya.~Khinchin: {\it Continued Fractions}, University of Chicago Press, 1964.

\bibitem{PSZ16}
E.~Pelantov\'a, \v S.~Starosta, M.~Znojil: Markov constant and quantum instabilities, {\it J. Phys. A: Math. Theor.} {\bf 49}
(2016), 155201.

\bibitem{Pe1913}
O.~Perron: {\it Die Lehre von den Kettenbr\"{u}chen}, B.~G.~Teubner, Leipzig, 1913.

\bibitem{Ro47}
R.M.~Robinson: Unsymmetrical approximation of irrational numbers, {\it Bull. Amer. Math. Soc.} {\bf 53} (1947), 351--361.

\bibitem{Sch80}
W.M.~Schmidt: {\it Diophantine Approximation}, Lecture Notes in Mathematics, vol.~785, Springer Verlag, Berlin 1980.

\bibitem{Se45}
B.~Segre: Lattice points in infinite domains, and asymmetric Diophantine approximations, {\it Duke Math. J.} {\bf 12} (1945), 337--365.

\bibitem{Sk79}
M.M.~Skriganov: Proof of the Bethe-Sommerfeld conjecture in
dimension two, {\it Soviet Math. Dokl.} {\bf 20} (1979),
956--959.

\bibitem{BS33}
A. Sommerfeld, H. Bethe: {\it Electronentheorie der Metalle}. 2nd
edition, Handbuch der Physik, Springer Verlag 1933.

\bibitem{Tu18}
O.~Turek: Gaps in the spectrum of a cuboidal periodic lattice graph, {\it Rep. Math. Phys.}, to appear (\texttt{arXiv:1801.02572}).

\bibitem{Zu35}
E.~Zurl: Theorie der reduziert-regelm\"{a}{\ss}igen Kettenbr\"{u}che, {\it Math. Ann.} {\bf 110} (1935), 679--717.

\end{thebibliography}
\end{document}